\documentclass[preprint,12pt,3p]{elsarticle}
\usepackage{graphicx}
\usepackage{amsmath}
\usepackage{amsfonts}
\usepackage{amssymb}

\newtheorem{thm}{Theorem}[section]
\newtheorem{cor}[thm]{Corollary}
\newtheorem{lem}[thm]{Lemma}

\newproof{proof}{Proof}
\newtheorem{pps}[thm]{Proposition}

\def\<{\langle}
\def\>{\rangle}
\begin{document}

\begin{frontmatter}

\title{Interplay between complex symmetry and Koenigs eigenfunctions}

\author{S. Waleed Noor\corref{cor1}}
\ead{waleed@ime.unicamp.br}
\author{Osmar R. Severiano}
\address{IMECC, Universidade Estadual de Campinas, Campinas-SP, Brazil.}

\cortext[cor1]{Corresponding author. Telephone: +55 19 992360945.}

\ead{osmar.rrseveriano@gmail.com}

\begin{abstract}
We investigate the relationship between the complex symmetry of composition operators $C_{\phi}f=f\circ \phi$ induced on the classical Hardy space $H^2(\mathbb{D})$ by an analytic self-map $\phi$ of the open unit disk $\mathbb{D}$ and its Koenigs eigenfunction. A generalization of orthogonality known as \emph{conjugate-orthogonality} will play a key role in this work. We show that if $\phi$ is a Schr\"{o}der map (fixes a point $a\in \mathbb{D}$ with $0<|\phi'(a)|<1$) and $\sigma$ is its Koenigs eigenfunction, then $C_{\phi}$ is complex symmetric if and only if $(\sigma^n)_{n\in \mathbb{N}}$ is complete and conjugate-orthogonal in $H^2(\mathbb{D})$. We study the conjugate-orthogonality of Koenigs sequences with some concrete examples. We use these results to show that commutants of complex symmetric composition operators with Schr\"{o}der symbols consist entirely of complex symmetric operators.
\end{abstract}


\end{frontmatter}

\section{Introduction}
Let $X$ be a vector space of analytic functions on the open unit disk $\mathbb{D}$ and $\phi$ an analytic self-map of $\mathbb{D}$. The \textit{composition operator} $C_\phi$ with \textit{symbol} $\phi$ is defined as
\begin{align*}
C_{\phi}f=f\circ \phi \hspace{0.3cm}  \textit{for}  \  \ f\in X.
\end{align*} 
 Operators of this type have been studied on a variety of spaces and in great detail. The objective is to study the interaction between the operator theortic properties of $C_\phi$ and function theoretic properties of $\phi$. It is well-known that $C_\phi$ is bounded on the classical Hardy Hilbert space $H^2(\mathbb{D})$ whenever $\phi$ is an analytic self-map of $\mathbb{D}.$ The monographs of Shapiro \cite{Shapiro text} and Cowen and McCluer \cite{Cowen text} contain detailed accounts of the subject. An analytic self map $\phi$ of $\mathbb{D}$ is called a \textit{Schr\"{o}der map} if it fixes a point $a\in \mathbb{D}$  with $0<|\phi'(a)|<1$. Let $H(\mathbb{D})$ denote Frechet space of all analytic functions on $\mathbb{D}.$ In 1884, Koenigs \cite{Koenigs} showed that if $\phi$ is a Schr\"{o}der map with fixed point $a\in \mathbb{D}$, then
the eigenvalues of the operator $C_{\phi}:H(\mathbb{D}) 	\rightarrow H(\mathbb{D}) $ are 
\begin{align*}
1, \phi'(a), \phi'(a)^2, \phi'(a)^3, \ldots.
\end{align*}

These eigenvalues are \emph{simple} (that have multiplicity one) and the eigenfunction $\sigma$ corresponding to the eigenvalue $\phi'(a)$ is called the \textit{Koenigs eigenfunction} for $\phi$. It follows that $\sigma^n$ is an eigenfunction corresponding to $\phi'(a)^n$ for each $n\in\mathbb{N}$, and all such eigenfunctions are scalar multiples of $\sigma^n$.  The sequence $(\sigma^n)_{n\in\mathbb{N}}$ will be referred to as the \emph{Koenigs sequence} for $\phi$.  The eigenfunction $\sigma$ does not necessarily belong to any of the Hardy spaces $H^p(\mathbb{D})$ for $p>0$. Characterizations of $H^p(\mathbb{D})$-membership of $\sigma$ have been obtained by Bourdon and Shapiro \cite{Paul1, Paul} and Poggi-Corradini \cite{Poggi}. In particular, they show that $\sigma\in H^p(\mathbb{D})$ for all $p>0$, or equivalently $\sigma^n\in H^2(\mathbb{D})$ for all $n\in\mathbb{N}$ if and only if the composition operator $C_\phi$ is a \emph{Riesz operator}, that is when the \emph{essential spectrum} of $C_\phi$ is the singleton $\{0\}$.

On the other hand a bounded operator $T$ on a separable Hilbert space $\mathcal{H}$ is \emph{complex symmetric} if $T$ has a self-transpose matrix representation with respect to some orthonormal basis for $\mathcal{H}$. An equivalent definition also exists. A \emph{conjugation} is a conjugate-linear operator $J:\mathcal{H}\to\mathcal{H}$ that satisfies the conditions \\

(a) $J$ is \emph{isometric}: $\langle Jf,Jg \rangle=\langle g,f \rangle$ $\forall$ $f,g\in\mathcal{H}$,\\

(b) $J$ is \emph{involutive}: $J^2=I$.\\ \\
 Then $T$ is $J$-\emph{symmetric} if $JT=T^*J$, and is called \emph{complex symmetric} if $T$ is $J$-symmetric for some conjugation $J$ on $\mathcal{H}$. A sequence $(f_n)_{n\in\mathbb{N}}$ is \emph{conjugate-orthogonal} in $\mathcal{H}$ if there exists a conjugation $J$ on $\mathcal{H}$ such that $\langle Jf_n,f_m\rangle=0$ for all $n\neq m$. An orthonormal basis $(e_n)_{n\in\mathbb{N}}$ for $\mathcal{H}$ is always conjugate-orthogonal. Indeed, just define the conjugation $Je_n=e_n$ for all $n\in\mathbb{N}$ and extend to all of $\mathcal{H}$ by conjugate-linearity. Hence $\langle Je_n,e_m\rangle=\delta_{m,n}$ where $\delta_{m,n}$ is the Kronecker delta. Complex symmetric operators are natural generalizations of complex symmetric matrices and normal operators, and their general study was initiated by
Garcia, Putinar, and Wogen (\cite{Garc2},\cite{Garc3},\cite{Wogen1},\cite{Wogen2}).  Conjugate-orthogonal systems were studied by Garcia and Putinar (\cite{Garc4}\cite{pamona}) as eigenvectors of complex symmetric operators.

The study of complex symmetric composition operators on $H^2(\mathbb{D})$ was initiated by Garcia and Hammond \cite{Garc1}. They showed that involutive disk automorphisms induce \emph{non-normal} complex symmetric composition operators. This was the first such example. Bourdon and Noor \cite{wal} proved that if $C_\phi$ is complex symmetric then $\phi$ must necessarily fix a point in $\mathbb{D}$. They further showed that if $\phi$ is a disk automorphism that is not elliptic of order two (involutive) or three, then $C_\phi$ is never complex sysmmetric.  Then Narayan, Sievewright and Thompson \cite{Narayan} discovered the first non-automorphic composition operators that are complex symmetric. Recently Narayan, Sievewright, Tjani \cite{Narayan Tjani} characterized all non-automorphic linear factional $\phi$ for which $C_\phi$ is complex symmetric. The analogous problem for the Hardy space of the half-plane $H^2(\mathbb{C_+})$ was recently solved by Noor and Severiano \cite{Noor Osmar}. For more general symbols the problem still remains open. It follows that the non-automorphic symbols $\phi$ of  interest from the point of view of complex symmetry are those that fix a point in $\mathbb{D}$, and that are also univalent (see \cite[Prop. 2.5]{Garc1}). These $\phi$ are precisely the \emph{univalent Scr\"{o}der maps} and the Koenigs eigenfunction $\sigma$ for $\phi$ is also univalent. In this case if $C_\phi$ is complex symmetric then it is a Riesz operator and the Koenigs sequence $(\sigma^n)_{n\in\mathbb{N}}$ is contained in $H^2(\mathbb{D})$ (\cite[Prop. 2.7]{Garc1}). 

The objective of this article is to study properties of the Koenigs sequence $(\sigma^n)_{n\in\mathbb{N}}$ and how they interact with the complex symmetry of $C_\phi$. The main result of this article can now be stated as follows (see Theorem \ref{Main 1}).

\begin{thm}\label{Main Theorem} Let $\phi$ be a Schr\"{o}der map with Koenigs eigenfunction $\sigma$. Then $C_{\phi}$ is complex symmetric if and only if the Koenigs sequence $(\sigma^n)_{n\in \mathbb{N}}$ is complete and conjugate-orthogonal in $H^2(\mathbb{D})$.
\end{thm}

A \emph{Riesz basis} is similar to an orthonormal basis (ONB), that is, it is the image of an ONB under an invertible operator. But Riesz bases are not necessarily conjugate-orthogonal. Let $\phi_a(z)=\frac{a-z}{1-\bar{a}z}$ be the involutive disk automorphism for some $a\in\mathbb{D}$. Then $(\phi_a^n)_{n\in\mathbb{N}}$ is a Riesz basis and is the Koenigs sequence for the Schr\"{o}der map \[\Phi_{a,\lambda}(z)=\phi_a(\lambda\phi_a(z))\] with fixed point $a$ and $\lambda\in\mathbb{D}$. We then apply Theorem \ref{Main Theorem} to show that this Riesz basis $(\phi_a^n)_{n\in\mathbb{N}}$ is \emph{not} conjugate-orthogonal (see Theorem \ref{Riesz basis not CS}).

\begin{thm} For $a,\lambda\in\mathbb{D}\setminus\{0\}$, the composition operator $C_{\Phi_{a,\lambda}}$ is not complex symmetric on $H^2(\mathbb{D})$, and therefore $(\phi_a^n)_{n\in\mathbb{N}}$ is not conjugate-orthogonal.
	\end{thm}

We next provide an example of a complete and conjugate-orthogonal sequence that is \emph{not} an ONB in $H^2(\mathbb{D})$ (see Theorem \ref{21}). 

\begin{thm}\label{z-a} The sequence $((z-a)^n)_{n\in\mathbb{N}}$ is complete and conjugate-orthogonal for all $a\in\mathbb{D}$, and is not an ONB when $a\neq 0$.
\end{thm}
If $\phi(z)=cz+d$ is an affine self-map of $\mathbb{D}$ with fixed point $a\in\mathbb{D}$, then $((z-a)^n)_{n\in\mathbb{N}}$ is the Koenigs sequence for $\phi$. Therefore by Theorems \ref{Main Theorem} and \ref{z-a}  we get a simple proof for one of the main results of Narayan, Sievewright and Thompson \cite{Narayan} (see Corollary \ref{24}).

\begin{cor} If $\phi$ is an affine self-map with a fixed point in $\mathbb{D}$, then $C_\phi$ is complex symmetric on $H^2(\mathbb{D})$.
	\end{cor}

The set of bounded operators that commute with an operator $T$ is called the \textit{commutant} of $T$ and is denoted by $T'$. Our main result about commutants of complex symmetric composition operators is the following (see Theorem \ref{commutant main}).

\begin{thm} Let $\phi$ be a Schr\"{o}der map such that $C_{\phi}$ is complex symmetric. Then each $A\in C_{\phi}'$ is also complex symmetric.
\end{thm}	

This is not true in general. If $\phi$ is an elliptic automorphism of order $2$, then $C_\phi$ is complex symmetric (see \cite{Garc1} and \cite{Noor}), but $C_{\phi}'$ contains composition operators that are not complex symmetric. On the other hand if $T$ is complex symmetric, then $T^n$ for $n\in\mathbb{N}$ are also complex symmetric. For composition operators with Schr\"{o}der symbols we obtain a strong converse of this (see Corollary \ref{Powers}).

\begin{cor} Let $\phi$ be a Schr\"{o}der map.  Then $C_{\phi}$ is complex symmetric if and only if $C_{\phi}^n$ is complex symmetric for some integer $n\geq 2.$
\end{cor}

This is not true in general even for composition operators, since if $\phi$ is an elliptic automorphism of order $4$ then $C_\phi$ is not complex symmetric (see \cite[Prop. 3.3]{wal}), but $C_\phi^2$ is elliptic of order two and is hence complex symmetric.

\section{Preliminaries}

\subsection{The Hardy space $H^2(\mathbb{D})$}The Hardy space $H^2:=H^2(\mathbb{D})$ is the space of all analytic functions in $\mathbb{D}$ given by $f(z)=\sum_{n=0}^{\infty}\widehat{f}(n)z^n,$ for which the norm  
\begin{align*}
\left\| f\right\| _{2}^2:= \sum_{n=0}^{\infty}| \widehat{f}(n)| ^2<\infty.
\end{align*}
The Hardy space $H^2$ is a Hilbert space with the following inner product: 
\begin{align*}
\left\langle f,g\right\rangle =\sum_{n=0}^{\infty}\widehat{f}(n)\overline{\widehat{g}(n)}
\end{align*}
where $f,g\in H^2.$ For each non-negative integer $n$ and $a\in \mathbb{D},$ the evaluation of the $n$th derivative of $f\in H^2$ at $a$ is  $f^{(n)}(a)=\langle f,K^{(n)}_{a}\rangle$ where
\begin{align}\label{5}
K_{a}^{(n)}(z)=\sum_{m=n}^{\infty}\frac{m!}{(m-n)!}\overline{a}^{m-n}z^m. 
\end{align} 
 We write $f=f^{(0)}$ and hence $K_a(z)=K_a^{(0)}(z)=\frac{1}{1-\bar{a}z}$ is the usual \emph{reproducing kernel} for $H^2$ at $a\in\mathbb{D}$. Let $\phi$ be a Schr\"{o}der map with fixed point $a\in\mathbb{D}$ and $\sigma$ its Koenigs eigenfunction normalized with $||\sigma||=1$. If $C_\phi$ is complex symmetric, then Garcia and Hammond \cite[p.176]{Garc1} proved the following relation we will use
 \begin{align}\label{13}
 	\left| \sigma(0) \right| 
 	=\frac{|K_{a}^{(1)}(a)|}{\left\| K_{a}\right\| \parallel  K_{a}^{(1)}\parallel }  .
 	\end{align}
 	
 \subsection{Conjugate-orthogonal vectors}

Complex symmetric operators satisfy the following \textit{spectral symmetry
property}. Let $T$ be $J$-symmetric on a Hilbert space $\mathcal{H}$ and $\lambda\in\mathbb{C}$, then  
\begin{equation}\label{Spec symmetry}
f\in \mathrm{Ker}(T-\lambda I)\Longleftrightarrow Jf\in \mathrm{Ker}(T^*-\overline{\lambda}I).
\end{equation}
A sequence $(f_n)_{n\in \mathbb{N}}$ is conjugate-orthogonal in $\mathcal{H}$ if $\langle Jf_n,f_m\rangle=0$ for $n\neq m$ and some conjugation $J$ on $\mathcal{H}$. Alternatively we shall just say $(f_n)_{n\in \mathbb{N}}$ is $J$-orthogonal. Note that if $(f_n)_{n\in \mathbb{N}}$ is also complete in $\mathcal{H}$, then $\langle Jf_n, f_n\rangle \neq0$ for all $n\in\mathbb{N}$.  This implies that $(Jf_n)_{n\in\mathbb{N}}$ is the unique sequence (upto scalar multiples) that is biorthogonal to $(f_n)_{n\in\mathbb{N}}$  (see \cite[Lemma 3.3.1]{Christensen}. In \cite{pamona}, Garcia proved the following.
\begin{lem}\label{J orth eigenvector}The eigenvectors of a $J$-symmetric operator $T$ corresponding to distinct eigenvalues are $J$-orthogonal.
\end{lem}
\begin{proof} Let $Tf_{i}=\lambda_i f_i$ for $i=1,2$ where $\lambda_1\neq\lambda_2$. Then
\[
\overline{\lambda_1}\langle Jf_1,f_2\rangle=\langle JTf_1,f_2\rangle=\langle T^*Jf_1,f_2\rangle=\langle Jf_1,Tf_2\rangle=\overline{\lambda_2}\langle Jf_1,f_2\rangle
\]
implies that $\langle Jf_1,f_2\rangle=0$.
	\end{proof}

If $\phi$ is a Schr\"{o}der map with fixed point $a\in\mathbb{D}$ and $\lambda=\phi'(a)$, then the eigenvalues $(\lambda^n)_{n\in\mathbb{N}}$ of $C_\phi$ are all distinct since $|\lambda|<1$. Therefore Lemma \ref{J orth eigenvector} gives us

\begin{pps}\label{29}
	Let $\phi$ be a Schr\"{o}der map with Koenigs eigenfunction $\sigma.$
	If $C_{\phi}$ is complex symmetric then $(\sigma^n)_{n\in \mathbb{N}}$ is conjugate-orthogonal in $H^2$.
\end{pps}

 \section{Complex symmetry and Koenigs eigenfuctions}
In this section, we characterize the complex symmetry of $C_\phi$ with Schr\"{o}der symbol $\phi$ in terms of the conjugate-orthogonality of its Koenigs sequence $(\sigma^n)_{n\in \mathbb{N}}$.

\begin{thm}\label{Main 1}Let $\phi$ be a Schr\"{o}der map with Koenigs eigenfunction $\sigma$. Then $C_{\phi}$ is complex symmetric if and only if the Koenigs sequence $(\sigma^n)_{n\in \mathbb{N}}$ is complete and conjugate-orthogonal in $H^2$.
\end{thm}
\begin{proof} Let $a\in \mathbb{D}$ be the fixed point of $\phi$ and denote $\lambda=\phi'(a)$. First suppose $(\sigma^n)_{n\in \mathbb{N}}$ is complete and $J$-orthogonal. Then for $n,m\in\mathbb{N}$, we have
	\begin{align*}
	\langle (C_{\phi}^*J-JC_{\phi})\sigma^n , \sigma^m\rangle=\langle J\sigma^n,C_\phi\sigma^m\rangle-\bar{\lambda}^n\langle J\sigma^n,\sigma^m\rangle=(\bar{\lambda}^m-\bar{\lambda}^n)\langle J\sigma^n, \sigma^m\rangle.
	\end{align*}	
	Since $\langle J\sigma^n,\sigma^m\rangle=0$ for all $n\neq m$, the completeness of $(\sigma^n)_{n\in \mathbb{N}}$ implies that $C_{\phi}^*J=JC_{\phi},$ and therefore that $C_\phi$ is $J$-symmetric. Conversely suppose $C_\phi$ is $J$-symmetric. By Proposition \ref{29} we know that $(\sigma^n)_{n\in \mathbb{N}}$ is conjugate-orthogonal. Hence we only need to show that $(\sigma^n)_{n\in \mathbb{N}}$ is complete in $H^2$. For each $n\in\mathbb{N}$, the function $K^{(n)}_{a}$ can be written in the form
	\begin{align}\label{v_n}
	K^{(n)}_{a}=a_nv_n+a_{n-1}v_{n-1}+\ldots+a_0v_0,
	\end{align}
	where $v_j$ is an eigenvector for $C_{\phi}^*$ corresponding to the eigenvalue $\overline{\lambda}^j$, and the $a_j$ are scalars (see \cite[Prop. 2.6]{Garc1}). Now apply $J$ on both sides of \eqref{v_n} and note that each $Jv_j$ is a scalar multiple of $\sigma^j$ by spectral symmetry (see \eqref{Spec symmetry}) and the fact that the $\lambda^j$ are simple eigenvalues of $C_\phi$. We then get
	\[
	JK^{(n)}_{a}=c_n\sigma^n+c_{n-1}\sigma^{n-1}+\ldots+c_0\sigma.
	\]
	To conclude the completeness of $(\sigma^n)_{n\in \mathbb{N}}$, we note that if some $f\in H^2$ is orthogonal to $(\sigma^n)_{n\in \mathbb{N}}$ then $f\perp JK_a^{(n)}$ and hence $Jf\perp K_a^{(n)}$ for all $n\in\mathbb{N}$. But $(Jf)^{(n)}(a)=0$ for all $n$ implies $Jf\equiv 0$ and hence $f\equiv 0$. Therefore $(\sigma^n)_{n\in \mathbb{N}}$ is complete in $H^2$.
\end{proof}	
\section{A Riesz basis that is not conjugate-orthogonal}\label{2}
Let $\phi_a(z)=\frac{a-z}{1-\bar{a}z}$ be the involutive disk automorphism for some $a\in\mathbb{D}\setminus\{0\}$. The goal of this section is to prove that although $(\phi_a^n)_{n\in\mathbb{N}}$ is a Koenigs sequence and a Riesz basis, it is not conjugate-orthogonal. The sequence $(\phi_a^n)_{n\in\mathbb{N}}$ is a Riesz basis since $\phi_a^n=C_{\phi_a}z^n$ for $n\in\mathbb{N}$ and it is also the Koenigs sequence for the Schr\"{o}der map \[\Phi_{a,\lambda}(z)=\phi_a(\lambda\phi_a(z))\] since $C_{\Phi_{a,\lambda}}\phi_a=\lambda\phi_a$ with $\lambda\in\mathbb{D}\setminus\{0\}$ and fixed point $a$.  

\begin{thm}\label{Riesz basis not CS} For $a,\lambda\in\mathbb{D}\setminus\{0\}$, the composition operator $C_{\Phi_{a,\lambda}}$ is not complex symmetric on $H^2$, and therefore $(\phi_a^n)_{n\in\mathbb{N}}$ is not conjugate-orthogonal.
\end{thm}

\begin{proof} First let $0<a<1$. Observe that by \eqref{5} we have 
	\[
	K_{a}^{(1)}(z)=\sum_{n=1}^{\infty}na ^{n-1}z^n.
	\] 
	Since $\phi_a$ is inner we have $||\phi_a||=1$. Moreover, simple computations show that
\begin{align}\label{7}
\	| K_{a}^{(1)}(a)| =\frac{a}{(1-a^2)^2}\ \ \text{and} \ \parallel K_{a}^{(1)}\parallel =\frac{(a^2+1)^{1/2}}{(1-a^2)^{3/2}}.
\end{align}
Now suppose on the contrary that $C_{\Phi_{a,\lambda}}$ is complex symmetric. Then \eqref{13} with $\sigma=\phi_{a}$ and $\phi=\Phi_{a,\lambda}$ gives
\begin{align}\label{8}
|K_{a}^{(1)}(a)|=\left| \phi_{a}(0)\right| \left\| K_{a}\right\| \parallel  K_{a}^{(1)}\parallel.
\end{align} 
By combining \eqref{7} and \eqref{8}, we obtain
\[
\frac{a}{(a^2-1)^2}=\frac{a(1+a^2)^{1/2}}{(a^2-1)^2}
\]
which is possible only if $a=0.$ However, this contradicts the hypothesis that $a\in (0,1)$ and therefore $C_{\Phi_{a,\lambda}}$ is not complex symmetric. For the general case $a\in\mathbb{D}\setminus\{0\}$, let $\theta$ be a real number such that $a=\left|a\right|e^{i\theta}$. Then 
\begin{equation}\label{9} 
e^{-i\theta}\left( \phi_{a}\left( e^{i\theta}z\right) \right) =e^{-i\theta}\left( \frac{a-e^{i\theta}z}{1-\overline{a}e^{i\theta}z}\right) =\frac{\left| a\right| -z}{1-\left| a\right| z}=\phi_{\left| a\right|}(z).
\end{equation}
By using \eqref{9}, we can express $\Phi_{a,\lambda}$ in term of $\Phi_{|a|,\lambda}$ 
\begin{align*}	e^{-i\theta}\Phi_{a,\lambda}(e^{i\theta}z)
&=e^{-i\theta}\phi_{\alpha}(\lambda\phi_{\alpha}(e^{i\theta}z))\nonumber=e^{-i\theta}\phi_{\alpha}(\lambda e^{i\theta}e^{-i\theta}\phi_{\alpha}(e^{i\theta}z))\nonumber\\
&=e^{-i\theta}\phi_{\alpha} (e^{i\theta}\lambda \phi_{|\alpha|}(z))=\phi_{|\alpha|}(\lambda\phi_{|\alpha|}(z))\\
&=\Phi_{|a|,\lambda}(z).
\end{align*}
This implies that $C_{e^{i\theta}z}C_{\Phi_{a,\lambda}}C_{e^{i\theta}z}^*=C_{\Phi_{|a|,\lambda}}$. Since $C_{e^{i\theta}z}$ is a unitary operator on $H^2$, it follows that $C_{\Phi_{a,\lambda}}$ is unitarily equivalent to $C_{\Phi_{|a|,\lambda}}$ and therefore $C_{\Phi_{a,\lambda}}$ is not complex symmetric. Finally since $(\phi_a^n)_{n\in\mathbb{N}}$ is the Koenigs sequence for $\Phi_{a,\lambda}$ and is complete in $H^2$, by Theorem \ref{Main 1} it follows that $(\phi_a^n)_{n\in\mathbb{N}}$ is not conjugate-orthogonal.
\end{proof}

\section{A complete and conjugate-orthogonal Koenigs sequence}
Let $a\in\mathbb{D}$ and $\sigma(z)=z-a$. In this section we show that $(\sigma^n)_{n\in\mathbb{N}}$ is a complete and conjugate-orthogonal sequence that is not an ONB when $a\neq 0$. First recall that for $r\in \mathbb{R}$ and $j\in \mathbb{N},$ the generalized binomial coefficient is defined as
\[
\binom{r}{j}=
\frac{r(r-1)\ldots (r-(j-1))}{j!}
\]
if $j>0$ and $\binom{r}{j}=1$ if $j=0$.
If $r\in\mathbb{N}$ with $0\leq j\leq r$, then $\binom{r}{j}=\frac{r!}{j!(r-j)!}$ is the usual binomial coefficient. In general, the binomial series for $(1-z)^r$ is 
\begin{align}\label{19}
(1-z)^r=\sum_{j=0}^{\infty}{\binom{r}{j}}(-z)^j
\end{align}
for each $z\in \mathbb{D}$. We begin by showing that $(\sigma^n)_{n\in\mathbb{N}}$ has a biorthogonal sequence $(z^nK_a^{n+1})_{n\in\mathbb{N}}$ when $a\in(-1,1)$. 
\begin{lem}\label{biorthogonal}
For each $a\in (-1,1)$ and non-negative integers $n$ and $m,$ we have
\begin{align}\label{20}
\langle z^nK_a^{n+1},(z-a)^m\rangle=\delta_{nm},
\end{align}
where $\delta$ is the Kronecker delta function.
\end{lem}
\begin{proof} First observe that $K^{n+1}_a(z)=(1-az)^{-(n+1)}.$ By using \eqref{19}, we get
\begin{equation}\label{znKa}
z^nK_a^{n+1}(z)=\sum_{j=0}^{\infty}\binom{-n-1}{j}(-a)^jz^{j+n}.
\end{equation}	
Then \eqref{20} holds for $m\leq n$ since the smallest term in \eqref{znKa} is $z^n$, whereas the largest term in $(z-a)^m$ is $z^m$. Therefore let $m> n$ and defining $k=m-n$ we get
\begin{align*}
\langle z^nK_a^{n+1}, (z-a)^m\rangle&=\langle \sum_{j=0}^{\infty}\binom{-n-1}{j}(-a)^jz^{j+n},\sum_{j=0}^m\binom{m}{j}(-a)^{m-j}z^j\rangle\\
&=\sum_{j=0}^{m-n}\binom{-n-1}{j}\binom{m}{n+j}(-a)^{m-n}\\
&=(-a)^{k}\sum_{j=0}^{k}\frac{(-n-1)(-n-2)\cdots (-n-j)}{j!}\frac{m!}{(n+j)!(k-j)!}\\
&=(-a)^{k}\sum_{j=0}^{k}\frac{(-1)^j(n+1)(n+2)\cdots (n+j)}{j!}\frac{m!}{(n+j)!(k-j)!}\\
&=(-a)^{k}\sum_{j=0}^{k}\frac{(-1)^j(n+j)!}{j!n!}\frac{m!}{(n+j)!(k-j)!}\\
&=(-a)^k\frac{m!}{n!k!}\sum_{j=0}^k(-1)^j\binom{k}{j}=0.	\end{align*}
The last sum vanishes by the identity $\sum_{j=0}^k(-1)^j\binom{k}{j}=0$ for $k>0$ \cite[Cor. 2]{Ruiz}. 
\end{proof}

Let $J$ be the conjugation on $H^2$ defined by $(Jf)(z)=\overline{f(\bar{z})}$,  $T_a$ the operator of multiplication by the normalized reproducing kernel $ K_{a}/\left\| K_{a}\right\|$ and $\phi_a(z)=\frac{a-z}{1-\bar{a}z}$. If $a\in(-1,1)$, then the operator defined by
\[
J_a=JT_a C_{\phi_a}
\]
 is a conjugation on $H^2$ by \cite[Prop. 2.3]{Narayan}. We arrive at the main result. 

\begin{thm}\label{21}
If $a \in \mathbb{D}$, then $((z-a)^n)_{n\in \mathbb{N}}$ is complete and conjugate-orthogonal in $H^2$. It is not an ONB when $a\neq 0$. 
\end{thm}
\begin{proof} First suppose $a\in(-1,1)$. It is easy to see that \[\phi_{a}(z)- a=-(1-a^2)zK_a(z)=-zK_a(z)/||K_a||^2.\] Then we get
\begin{align}\label{22}
J_{a}(z-a)^n&=K_a(z)(\phi_a(z)-a)^n/||K_a||=(-1)^n||K_a||^{-2n-1}z^nK^{n+1}_{a}(z).
\end{align}
Therefore $J_a(z-a)^n$ is a scalar multiple of $z^nK_a^{n+1}$ for $n\in\mathbb{N}$ and hence $(z-a)^n$ is $J_a$-orthogonal by Lemma \ref{biorthogonal} for $a\in(-1,1)$. For the general case $a\in \mathbb{D}$, let $\theta$ be a real number such that $a=e^{i\theta}|a|$. Since $J_{|a|}$ is a conjugation for each $a\in \mathbb{D}$ and $U_{\theta}=C_{e^{i\theta}z}$ is a unitary operator, it follows that $U_{\theta}^*J_{|a|}U_{\theta}$ is also a conjugation. Therefore we get
\begin{align*}
\langle U_{\theta}^*J_{|a|}U_{\theta}(z-a)^n, (z-a)^m\rangle&=\langle J_{|a|}U_{\theta} (z-a)^n, U_{\theta}(z-a)^m\rangle\\
&=\langle J_{|a|}(e^{i\theta}z-a)^n,(e^{i\theta}z-a)^m\rangle\\
&=e^{-i\theta(n+m)}\langle J_{|a|}(z-|a|)^n, (z-|a|)^m\rangle   
\end{align*}	
for $m,n\in\mathbb{N}$. Now since $(z-|a|)^n$ is $J_{|a|}$-orthogonal by the first part, it follows that $((z-a)^n)_{n\in\mathbb{N}}$ is $U_{\theta}^*J_{|a|}U_{\theta}$-orthogonal. We now prove the completeness of $((z-a)^n)_{n\in\mathbb{N}}$. Note that the operator $U:f(z)\to f(z-a)$ is unbounded on $H^2$ and has the polynomials $\mathbb{C}[z]$ in its domain. By $Uz^n=(z-a)^n$ for $n\in\mathbb{N}$ it is enough to show that $U\mathbb{C}[z]=\mathbb{C}[z]$. This is clear since for any polynomial $p$ we have $U(p(z+a))=p(z)$ where $p(z+a)$ is also a polynomial. The non-orthonormality is clear since $1\not\perp(z-a)^n$ for $n>0$ when $a\neq 0$.
\end{proof}

As an immediate application of Theorem \ref{21} we obtain one of the main results of  Narayan, Sievewright and Thompson \cite{Narayan}. 
\begin{cor}\label{24}
If $\phi$ is an affine self-map with a fixed point in $\mathbb{D}$, then $C_\phi$ is complex symmetric on $H^2$.
\end{cor}
\begin{proof} Let $\phi(z)=cz+d$ with fixed point $a=\frac{d}{1-c}\in\mathbb{D}$. Then we see that 
	\[
	C_\phi(z-a)^n=c^n(z-a)^n
	\]
	where $|c|<1$ necessarily and hence $((z-a)^n)_{n\in\mathbb{N}}$ is the Koenigs sequence for $\phi$. Therefore $C_\phi$ is complex symmetric by Theorems \ref{Main 1} and \ref{21}.
	\end{proof}
\section{Commutants of complex symmetric composition operators}\label{4}
In this section, we study how the complex symmetry of $C_{\phi}$ affects the structure of its commutant $C_{\phi}'$ when $\phi$ is a Schr\"{o}der symbol. The main result is the following.

\begin{thm}\label{commutant main} Let $\phi$ be a Schr\"{o}der map such that $C_{\phi}$ is complex symmetric on $H^2$. Then each $A\in C_{\phi}'$ is complex symmetric.
	\end{thm}
\begin{proof} Let $a\in \mathbb{D}$ be the fixed point of $\phi$, $\sigma$ its Koenigs eigenfunction and $\lambda=\phi'(a)$. If $A$ commutes with $C_{\phi}$ then 
	$
	C_{\phi}A\sigma^n=\lambda^nA\sigma^n
	$
	 for each $n\in\mathbb{N}$. By Koenigs' Theorem there must exist $a_n\in\mathbb{C}$ such that $A\sigma^n=a_n\sigma^n$ for $n\in\mathbb{N}$. If $C_\phi$ is $J$-symmetric, then for $n,m\in\mathbb{N}$ we have
	 \[
	 \left\langle \sigma^m, (A^*-\overline{a_n}I)J\sigma^n \right\rangle =\left\langle  (A-a_n I)\sigma^m, J\sigma^n\right\rangle 
	 =(a_m-a_n) \left\langle \sigma^m, J\sigma^n \right\rangle.
	 \]
	 Since $(\sigma^n)_{n\in\mathbb{N}}$ is complete $J$-orthogonal by Theorem \ref{Main 1}, we obtain $A^*J\sigma^n=\overline{a_n}J\sigma^n.$ Hence
	 \begin{align*}
	 A^*J\sigma^n=\overline{a_n}J\sigma^n=J(a_n\sigma^n)=JA\sigma^n
	 \end{align*}
	 and the completeness of $(\sigma^n)_{n\in\mathbb{N}}$ implies $A^*J=JA$. 
	\end{proof}

The conclusion of Theorem \ref{commutant main} does not hold in general. Let $\phi=\phi_{a}\circ(-\phi_{a})$ with $a\in \mathbb{D}\backslash\{0\}$. Then $\phi$ is elliptic of order $2$ and $C_\phi$ is complex symmetric. If $\Phi_{a,\lambda}=\phi_{a}\circ(\lambda \phi_{a})$ where $\lambda\in \mathbb{D}\backslash\{0\}$, then $\Phi_{a,\lambda}$ commutes with $\phi$ and hence $C_{\Phi_{a,\lambda}}\in C_{\phi}'$ even though $C_{\Phi_{a,\lambda}}$ is \emph{not} complex symmetric by Theorem \ref{Riesz basis not CS}. An immediate implication of Theorem \ref{commutant main} is the following interesting result.

\begin{cor}\label{Powers} Let $\phi$ be a Schr\"{o}der map.  Then $C_{\phi}$ is complex symmetric if and only if $C_{\phi}^n$ is complex symmetric for some integer $n\geq 2$.
\end{cor}
\begin{proof} If $C_\phi$ is complex symmetric then clearly so is $C_\phi^n$ for all $n\in\mathbb{N}$. Conversely suppose $C_{\phi}^n$ is complex symmetric for some $n\geq 2$.  Clearly the $n$-th composite $\phi^{[n]}$ of $\phi$ is also a Schr\"{o}der map. Since $C_{\phi^{[n]}}=C_{\phi}^n$ is complex symmetric and $C_{\phi}$ commutes with $C_{\phi^{[n]}},$ it follows from Theorem \ref{commutant main} that $C_{\phi}$ is complex symmetric.
\end{proof}

Again this is not true in general, since if $\phi$ is an elliptic automorphism of order $4$ then $C_\phi$ is \emph{not} complex symmetric (see \cite[Prop. 3.3]{wal}), but $C_\phi^2$ is elliptic of order two and is hence complex symmetric.

\section*{Acknowledgement}
This work constitutes a part of the doctoral dissertation of the second author under the supervision of the first author.

\bibliographystyle{amsplain}

\end{document}